\documentclass[10pt]{amsart}
\usepackage{amsmath}
\usepackage[usenames,dvipsnames]{color}
\usepackage{parskip}
\usepackage{amsfonts}
\usepackage{amscd}
\usepackage[centertags]{amsmath}
\usepackage{amssymb}
\usepackage[all,cmtip]{xy}
\usepackage[english]{babel}
\usepackage{tabularx}
\usepackage{mathtools}
\usepackage{amsxtra}
\usepackage{euscript}
\usepackage[T1]{fontenc}
\usepackage{doc, exscale, fontenc, latexsym, syntonly}
\usepackage{amsfonts}
\usepackage{amsthm}
\usepackage{graphicx}
\usepackage{xcolor}
\usepackage{tikz-cd}

\numberwithin{figure}{section}
\numberwithin{table}{section}

\newcommand{\cat}{\ensuremath{\mathrm{cat}}}
\newcommand{\scat}{\ensuremath{\mathrm{scat}}}

\newcommand{\id}{\ensuremath{\mathrm{id}}}
\newcommand{\TC}{\ensuremath{\mathrm{TC}}}
\newcommand{\D}{\ensuremath{\mathrm{D}}}

\newcommand{\SD}{\ensuremath{\mathrm{SD}}}
\newcommand{\sd}{\ensuremath{\mathrm{sd}}}

\newtheorem{theorem}{Theorem}[section]
\newtheorem{definition}{Definition}[section]
\newtheorem{corollary}{Corollary}[section]
\newtheorem{remark}{Remark}[section]
\newtheorem{example}{Example}[section]
\newtheorem{proposition}{Proposition}[section]
\newtheorem{lemma}{Lemma}[section]

\begin{document}

\title{Higher Contiguity Distance}

\author{N\.{I}lay Ek\.{I}z Yazici\textsuperscript{1}, Ayşe Borat\textsuperscript{1,2}}

\date{\today}

\address{\textsc{Nilay Ekiz Yazıcı}
Bursa Technical University\\
Faculty of Engineering and Natural Sciences\\
Department of Mathematics\\
Bursa, Turkiye}
\email{nilay.ekiz@btu.edu.tr} 

\address{\textsc{Ayşe Borat}
Bursa Technical University\\
Faculty of Engineering and Natural Sciences\\
Department of Mathematics\\
Bursa, Turkiye}
\email{ayse.borat@btu.edu.tr}



\begin{abstract} In this paper, we introduce the higher analogues of contiguity distance and its relations with simplicial Lusternik-Schnirelmann category and discrete topological complexity. Also we study the effects of geometric realisation and barycentric subdivision in the sense that how the geometric realisation of the simplicial maps and the induced simplicial maps on barycentric subdivisions affects higher contiguity distance. 
\end{abstract}

\maketitle

\footnotetext[1] {The authors are supported by the Scientific and Technological Research Council of Turkey (T\"{U}B\.{I}TAK) [grant number 123F268].}
\footnotetext[2]{The corresponding author.}

\section{Introduction}

Homotopic distance (\cite{MVML}) is a mathematical tool which can be realised as a generalisation of the homotopy invariants topological complexity (\cite{F2}) and Lusternik-Schnirelmann category (\cite{CLOT}). All these invariants are important since that they can be used to analysize motion planning problem. Moreover, $n$-th topological complexity is related to the motion planning problem which requires a rule to assign any $n$ points in the space to an $n$-legged path. The concept of higher topological complexity is introduced by Rudyak in \cite{R} and is developed by Basabe, Gonzalez, Rudyak and Tamaki in \cite{BGRT}. Motivated from that fact, in \cite{MVML} Macias-Virgos, Mosquera-Lois and in \cite{BV} Borat, Vergili construct $n$-th homotopic distance in a way that for some special continuous maps it coincides with $n$-th topological complexity and Lusternik-Schnirelmann category.

One of the importance of considering motion planning problem in the simplicial realm is that it gives a more computable approach on motion planning problem. Regarding the problem in a combinatorial way leads one to ask how the ideas in the preceding paragraph are modelled in a combinatorial way. The answer comes from Fernandez-Ternero, Macias-Virgos, Minuz and Vilches for discrete topological complexity in \cite{FMMV}, from the same authors for simplicial Lusternik-Schnirelmann category in \cite{FMMV2} and \cite{FMV}, from Borat, Pamuk and Vergili for simplicial analogue of homotopic distance so-called contiguity distance in \cite{BPV}, from Alabay, Borat, Cihangirli and Dirican Erdal for $n$-th discrete topological complexity in \cite{ABCD}.

Moreover, Gonzalez studied simplicial analogue of topological complexity from a different viewpoint in \cite{G}. 

In this paper, we turn our face towards contigutiy distance and focus on to build its higher analogues.


\section{Preliminaries}

In this section, we will recall some combinatorial definitions which will be used to construct the main tools of this paper.

\begin{definition}
Two simplicial maps $\varphi, \psi : K \to L$ are called contiguous if for  every simplex $\{v_0, \dots, v_k\}$ in $K$, $\{\varphi(v_0), \dots , \varphi(v_k),\psi(v_0), \dots , \psi(v_k)\}$ constitutes a simplex in $L$. Such maps are denoted by $\varphi \sim_c \psi$.
\end{definition}

\begin{definition}
Two simplicial maps $\varphi, \psi : K \rightarrow L $ are said to be in the same contiguity class if one can find a finite sequence of simplicial maps $\varphi_i : K \rightarrow L$ for $i=0,1,\dots m$ such that $\varphi = \varphi_1 \sim_c \varphi_2 \sim_c \dots \sim_c \varphi_m= \psi$. Such maps are denoted by $\varphi \sim \psi $.
\end{definition}

\begin{definition}
	Let $K$ be a simplicial complex. A finite or infinite sequence of vertices such that any two consecutive vertices span an edge is called an edge path in $K$. If any two vertices in $K$ can be joined by a finite edge path, $K$ is said to be edge-path connected. 
\end{definition}

\begin{lemma}\cite{FMMV}\label{soru} $K$ is edge-path connected simplicial complex iff any two constant simplicial maps $L\rightarrow K$ are in the same contiguity class.
\end{lemma}

Since the cartesian product of simplicial complexes is not necessarily a simplicial complex, a "new" product on simplicial complexes is defined so that two simplicial complexes under the new product constitutes a simplicial complex.

Suppose that $K_1$ and $K_2$ are simplicial complexes. The categorical product of $K_1$ and $K_2$ (denoted by $K_1\prod K_2$) is defined as follows.

\begin{itemize}
\item[(i)] The set of vertices of $K_1\prod K_2$ is defined to be $V(K_1)\times V(K_2)$ where $V(K_i)$ stands for the set of vertices of $K_i$ for $i=1,2$.
\item[(ii)] Let $p_i: V(K_1\prod K_2)\rightarrow V(K_i)$ be the projection maps for $i=1,2$. A simplex $\sigma$ is said to be in $K_1\prod K_2$, if $p_1(\sigma)\in K_1$ and $p_2(\sigma)\in K_2$. 
\end{itemize} (for more details, see \cite{K}).

The following concept can be realised as the simplicial analogue to the homotopy type of topological spaces.

\begin{definition} The simplicial complexes $K$ and $K'$ are said to have the same strong homotopy type if there exist simplicial maps $\varphi: K\rightarrow K'$ and $\psi: K'\rightarrow K$ such that $\varphi\circ \psi \sim \id_K$ and  $\psi\circ \varphi \sim \id_{K'}$, and is denoted by $K\sim K'$. Here, the simplicial maps $\varphi$ and $\psi$ are called strong homotopy equivalences.
\end{definition}

\begin{definition} Let $K$ be a simplicial complex and $v_0$ be a vertex in $K$. A simplicial complex $K$ is said to be strongly collapsible if $K$ and $\{v_0\}$ is of the same strong homotopy type.
\end{definition}

\begin{definition}\cite{MVML} For simplicial maps $\varphi, \psi:K\rightarrow K'$, the contiguity distance between $\varphi$ and $\psi$, denoted by $\SD(\varphi,\psi)$, is the least integer $k\geq 0$ such that there exists a covering of $K$ by subcomplexes $\Omega_0,\Omega_1,...,\Omega_k$ with the property that $\varphi|_{\Omega_j}$ and $\psi|_{\Omega_j}$ are in the same contiguity class for all $j=0,1,...,k.$ If there is no such a covering, it is defined to be $\SD(\varphi,\psi)=\infty$.
\end{definition}

\begin{definition} \cite{FMV}
Let $K$ be a simplicial complex. For a subcomplex $\Omega\subset K$, if the inclusion $i:\Omega\hookrightarrow K$ and a constant map $c_{v_0}: \Omega\rightarrow K$, where $v_0 \in K$ is some fixed vertex, are in the same contiguity class, then $\Omega$ is called categorical.
\end{definition}

\begin{definition} \cite{FMV} Let $K$ be a simplicial complex. The simplicial Lusternik-Schnirelmann category $\scat(K)$ is the least integer $k\geq 0$ such that one can find categorical subcomplexes $\Omega_0,\Omega_1,\dots, \Omega_k$ of $K$ covering $K$.
\end{definition}

\begin{definition}\cite{ABCD} \label{DefnFarber}
Let $K$ be a simplicial complex. For a subcomplex $\Omega\subset K^n$, if there is a simplicial map $\sigma : \Omega \to K$ such that $\Delta\circ \sigma$ $\sim \iota_{\Omega}$ where $\Delta : K \to K^n$, $\Delta(v) = (v,v,\dots,v)$ is the diagonal map and $\iota_{\Omega} : \Omega \hookrightarrow K^n$ is the inclusion map, then $\Omega$ is called an $n-$Farber subcomplex.
\end{definition}

\begin{definition}\cite{ABCD}\label{DefnDTC}
Let $K$ be a simplicial complex. The $n$-th discrete topological complexity $\TC_n(K)$ of a simplicial complex $K$ is the least integer $k\geq 0$ such that one can find $n-$Farber subcomplexes $\Omega_0,\Omega_1,\dots, \Omega_k$ of $K^n$ covering $K^n$.
\end{definition}

The following theorem gives a characterisation of $n$-Farber subcomplexes.

\begin{theorem} \cite{ABCD} \label{teo3.4}
Let $\Omega \subset K^n$ be a subcomplex and $\Delta:K \rightarrow K^n$ denote the diagonal map, then the followings are equivalent.

\begin{itemize}
    \item[(1)] $\Omega$ is an $n$-Farber subcomplex.  
    \item[(2)] $(p_i)_|{}_\Omega \sim (p_j)_|{}_\Omega$ for all $i,j\in\{1,2,\ldots,n\}$. 
    \item[(3)] One of the restrictions $(p_1)_|{}_\Omega, (p_2)_|{}_\Omega, \ldots, (p_n)_|{}_\Omega $ is a section of $\Delta$ (up to contiguity).
\end{itemize}   
\end{theorem}

If $n=2$ is in Definition~\ref{DefnFarber} and Definition~\ref{DefnDTC}, then one obtains Farber subcomplexes and discrete topological complexity, respectively, as introduced in \cite{FMMV}.

\section{Higher Contiguity Distance}

\begin{definition}
	For simplicial maps $\varphi_1,...\varphi_n:K\rightarrow K'$, the $n$-th contiguity distance between $\varphi_1,...\varphi_n:K\rightarrow K'$, denoted by $\SD(\varphi_1,...,\varphi_n)$, is the least integer $k\geq 0$ such that there exists a covering of $K$ by sobcomplexes $K_0,K_1,...,K_k$ with the property that $\varphi_1|_{K_j},\varphi_2|_{K_j},...,\varphi_n|_{K_j}:K_j\rightarrow K'$ are in the same contiguity class for all $j=0,1,...,k.$ If there is no such a covering, it is defined to be $\SD(\varphi_1,...,\varphi_m)=\infty$.
\end{definition}

\begin{proposition}
	$\SD(\varphi_1,...,\varphi_n)=\SD(\varphi_{\sigma(1)},...,\varphi_{\sigma(n)})$ holds for any permutation $\sigma$ of $\{1,...,n\}$.\qed
\end{proposition}

\begin{proposition}
	$\SD(\varphi_1,...,\varphi_n)=0$ if and only if $ \varphi_i \sim\varphi_{i+1}$ for each $i \in \{1,...,n-1\}$. \qed
\end{proposition}

\begin{proposition} \label{prop23}
	Given simplicial maps $\varphi_i:K\rightarrow K'$ and $\psi_i:K\rightarrow K'$ for $i\in \{1,...,n\}.$ If $\varphi_i\sim\psi_i$ for each i, then $\SD(\varphi_1,...,\varphi_n)=\SD(\psi_1,...,\psi_n).$\qed
\end{proposition}

\begin{proposition}\label{prop3.4}
	If $1<m<n$ and $\varphi_1,...,\varphi_m,...,\varphi_n:K\rightarrow K'$ are simplicial maps, then $\SD(\varphi_1,...,\varphi_m)\leq \SD(\varphi_1,...,\varphi_n).$\qed
\end{proposition}

\begin{example} Consider the simlicial complex $K$ given in Figure~\ref{f1}. Given the constant simplicial maps $\varphi_1, \varphi_2:K\rightarrow K$ at the vertices $\{0\}$ and $\{1\}$, respectively, and the simplicial map 
\[
\varphi_3:K\rightarrow K \hspace{0.05in} \text{by }
\]

\[
\varphi_3(\{0\})=\varphi_3(\{1\})=\{0\}, \hspace{0.05in} \varphi_3(\{2\})=\varphi_3(\{3\})=\varphi_3(\{4\})=\{1\}.
\]
\vspace{0.01in}

\begin{center}
\newcommand\size{1}
  
 \begin{tikzpicture}\label{f1}
     \def\size{2} 
    \draw[thick]  (18:\size) \foreach \a [count=\i] in {90,162,234,306}{ -- (\a:\size) } -- cycle;
    
    \draw[thick] (18:\size) \foreach \a [count=\i] in {162,306,90,234} { -- (\a:\size) } -- cycle;
    
    \foreach \i/\a in {0/18, 1/90, 2/162, 3/234, 4/306} {
        \node[black, fill=black, circle, inner sep=2pt, label={[label distance=-3pt] \a:\i}] at (\a:\size) {};
    }
    
\end{tikzpicture}
\end{center}   
\[
\textit{Figure 3.1}
\] 
	
	
	
	
\vspace{0.05in}
	
For $i,j=1,2,3$, $\varphi_i(v)\cup\varphi_j(v)$ is a simplex whenever $v$ is a vertex of $K$.

Moreover, we have the following.

$\varphi_1(\sigma)\cup\varphi_2(\sigma)$ is either $\{0\}$ or $\{1\}$ or $\{0,1\}$;\\
$\varphi_1(\sigma)\cup\varphi_3(\sigma)$ is either $\{0\}$ or $\{0,1\}$;\\
$\varphi_2(\sigma)\cup\varphi_3(\sigma)$ is either $\{1\}$ or $\{0,1\}$.
	
So one can conclude that $\varphi_1$, $\varphi_2$ and $\varphi_3$ are contiguous, so are in the same contiguity class. Therefore $\SD(\varphi_1,\varphi_2,\varphi_3)=0$.
\end{example}

\begin{proposition} \label{prop215}
	Let $\varphi_1,\varphi_2,...,\varphi_n:K\rightarrow K'$ and $\mu:M\rightarrow K$ be simplicial maps. Then we have 
	$$ \SD(\varphi_1\circ\mu,...,\varphi_n\circ\mu)\leq \SD(\varphi_1,...,\varphi_n). $$
\end{proposition}

\begin{proof} 
	Let $\SD(\varphi_1,...,\varphi_n)=k.$ Then there exist subcomplexes $K_0,...,K_k$ of $K$ such that $\varphi_1|_{K_i}\sim...\sim\varphi_n|_{K_i}$ for all $i$. Define $M_i:=\mu^{-1}(K_i)$. Then 
	$$(\varphi_s\circ\mu)|_{M_i}=\varphi_s|_{K_i}\circ \mu|_{M_i}=\varphi_s\circ\mu|_{M_i}\sim \varphi_t\circ\mu|_{M_i}=\varphi_t|_{K_i}\circ \mu|_{M_i}=(\varphi_t\circ\mu)|_{M_i}$$
	for all $s,t\in \{1,...,n\}$. Therefore,  $\SD(\varphi_1\circ\mu,...,\varphi_n\circ\mu)\leq k.$
\end{proof}

The following proposition is a generalisation of Proposition \ref{prop215}.

\begin{proposition} \label{prop216}
		Let $\varphi_1,\varphi_2,...,\varphi_n:K\rightarrow K'$ and $\mu_1,\mu_2,...,\mu_n:M\rightarrow K$ be simplicial maps. If $\mu_1\sim \mu_2 \sim ... \sim \mu_n $ then we have 
		$$ \SD(\varphi_1\circ\mu_1,...,\varphi_n\circ\mu_n)\leq \SD(\varphi_1,...,\varphi_n). $$
\end{proposition}

\begin{proof}
	Let $\SD(\varphi_1,...,\varphi_n)=k.$ Then there exist subcomplexes $K_0,...,K_k$ of $K$ such that $\varphi_1|_{K_i}\sim...\sim\varphi_n|_{K_i}$ for all $i$. Define $M_i:=\mu_1^{-1}(K_i)$. Then 
	$$(\varphi_s\circ\mu_{1})|_{M_i}=\varphi_s|_{K_i}\circ\mu_{1}|_{M_i}=\varphi_s\circ\mu_{1}|_{M_i}\sim \varphi_t\circ\mu_{t'}|_{M_i}=\varphi_t|_{K_i}\circ\mu_{t'}|_{M_i}=(\varphi_t\circ\mu_{t'})|_{M_i}$$
	for all $s,t,t' \in \{1,...,n\}.$ So $\SD(\varphi_1\circ\mu_1,...,\varphi_n\circ\mu_n)\leq k.$
\end{proof}

\begin{corollary}\label{corol21}
    Let $\mu_1,...,\mu_n: M\rightarrow K$ be the simplicial maps that satisfy the condition in Proposition~\ref{prop216} which also have right strong equivalence. Then for simplicial maps $\varphi_1,...,\varphi_n:K\rightarrow K'$, we have 
    $$\SD(\varphi_1\circ \mu_1,...,\varphi_n\circ\mu_n)=\SD(\varphi_1,...,\varphi_n).$$
\end{corollary}

\begin{proof}
	If $\mu_i$ has right strong equivalence (for all $i=1,...,n$), then we have $\mu_i\circ \alpha\sim id_K.$ Here, since $\mu_i\sim\mu_{i+1}$ for all i, the same $\alpha$ can be used for each $i$. It follows that $\varphi\circ\mu_i\circ\alpha\sim\varphi_i$. Thus,
	\begin{eqnarray*}
	\SD(\varphi_1,...,\varphi_n)&=&\SD(\varphi_1\circ\mu_1\circ\alpha,...,\varphi_n\circ\mu_n\circ\alpha)\\
	&\leq& \SD(\varphi_1\circ\mu_1,...,\varphi_n\circ\mu_n)\\
	&\leq& \SD(\varphi_1,...,\varphi_n)
    \end{eqnarray*}
    where the equality and the inequalities follow from Proposition \ref{prop23} and Proposition \ref{prop216}, respectively. Hence, we have $\SD(\varphi_1\circ \mu_1,...,\varphi_n\circ\mu_n)=\SD(\varphi_1,...,\varphi_n).$
\end{proof}

\begin{proposition} \label{prop217}
	Let $\varphi_1,...,\varphi_n:K\rightarrow K'$ and $\mu_1,...,\mu_n:K'\rightarrow M$ be simplicial maps. If $\mu_i \sim \mu_{i+1}$ for every $i\in\{1,...,n-1\}$ then $$ \SD(\mu_1\circ\varphi_1,...,\mu_n\circ\varphi_n)\leq \SD(\varphi_1,...,\varphi_n).$$
\end{proposition}

\begin{proof}
	Suppose that $\SD(\varphi_1,...,\varphi_n)=k.$ There exist subcomplexes $K_0,...,K_k$ of $K$ such that $\varphi_1|_{K_i}\sim...\sim\varphi_n|_{K_i}$ for all $i$. Then 
	$$ (\mu_s\circ\varphi_{s'})|_{K_i}=\mu_s\circ\varphi_{s'}|_{K_i}\sim\mu_s\circ\varphi_{t'}|_{K_i}\sim\mu_t\circ\varphi_{t'}|_{K_i}=(\mu_t\circ\varphi_{t'})|_{K_i}$$ for all $s,s',t,t'\in \{1,...,n\}.$ So $\SD(\mu_1\circ\varphi_1,...,\mu_n\circ\varphi_n)\leq k.$
\end{proof}

\begin{corollary}\label{corol22}
	Let $\mu_1,...,\mu_n:K'\rightarrow M$ be the simplicial maps that satisfy the condition in Proposition~\ref{prop217} which also have left strong equivalence. Then for simplicial maps $\varphi_1,...,\varphi_n:K\rightarrow K'$, we have
	 $$ \SD(\mu_1\circ\varphi_1,...,\mu_n\circ\varphi_n)=\SD(\varphi_1,...,\varphi_n). $$ 
\end{corollary}

\begin{proof}
	If $\mu_i$ has left strong equivalence  (for all $i=1,...,n$), then $\beta\circ\mu_i\sim id_K.$ Here, since $\mu_i\sim\mu_{i+1}$ for all i, the same $\beta$ can be used for each $i$. It follows that $\beta\circ\mu_i\circ\varphi_i\sim \varphi_i$. Thus,
	\begin{eqnarray*}
		\SD(\varphi_1,...,\varphi_n)&=& \SD(\beta\circ\mu_1\circ\varphi_1,...,\beta\circ\mu_n\circ\varphi_n) \\
		&\leq& \SD(\mu_1\circ\varphi_1,...,\mu_n\circ\varphi_n) \\
		&\leq& \SD(\varphi_1,...,\varphi_n)
	\end{eqnarray*}
	where the equality and the inequalities follow from Proposition~\ref{prop23} and Proposition~\ref{prop217}, respectively. Hence, we have $ \SD(\mu_1\circ\varphi_1,...,\mu_n\circ\varphi_n)=\SD(\varphi_1,...,\varphi_n). $
	
\end{proof}

As a result of Corollary~\ref{corol21} and Corollary~\ref{corol22}, one can write the following theorem.

\begin{theorem}\label{maininvariant} Let $\beta_1,\ldots,\beta_n,:K'\rightarrow K$ and $\alpha_1,\ldots, \alpha_n:L\rightarrow L'$ be right and left strong equivalences, respectively. If the simplicial maps $\varphi_1,\ldots, \varphi_n:K \rightarrow L$ and $\psi_1,\ldots, \psi_n:K' \rightarrow L'$ make the following diagram commutative up to contiguity for each $j=1,\ldots,n$ (i.e., $\alpha_j\circ\varphi_j\circ \beta_j \sim \psi_j$ for each $j$), then we have $\SD(\varphi_1,\ldots,\varphi_n)=\SD(\psi_1,\dots,\psi_n)$.
\begin{displaymath}
	\xymatrix{
		K \ar[r]^{\varphi_j}  &
		L \ar[d]^{\alpha_j} \\
		K' \ar[r]_{\psi_j} \ar[u]_{\beta_j} & L' }
	\end{displaymath}

\end{theorem}

\begin{remark} In particular, the above theorem is valid if one consider only one right strong equivalence $\beta$ and one left strong equivalence $\alpha$. Therefore one can say that if $K\sim K'$ and $L\sim L'$, the result of the above theorem is still valid and it can be realised as a theorem which mentions the strong homotopy invariance of the higher contiguity distance. 
\end{remark}

\begin{theorem}
	$\SD(p_1,...,p_n)=\TC_n(K)$ where $p_i:K^n\rightarrow K$ is the projection to the i-th factor for $i\in \{1,...,n\}$.
\end{theorem}

\begin{proof} 
	First of all, we show that $\SD(p_1,...,p_n)\leq \TC_n(K)$. Suppose that $\TC_n(K)=k$. Then there exist $n$-Farber subcomplexes $K_1,...,K_k$ covering $K^n$. From Theorem~\ref{teo3.4} it becomes $p_1|_{K_i}\sim...\sim p_n|_{K_i}$ for all i. So $\SD(p_1,...,p_n)\leq k.$ \\
	On the other hand, let $\SD(p_1,...,p_n)=k$. Then there exist subcomplexes $K_1,...,K_k$ such that $p_1|_{K_i}\sim...\sim p_n|_{K_i}$ for all i. By Theorem~\ref{teo3.4}, for each i, $K_i$ is $n$-Farber subcomplex. Therefore $\TC_n(K)\leq k.$
\end{proof}

\begin{corollary}\label{invariant} The $n-$th discrete topological complexity $\TC_n$ is an invariant of strong homotopy type.
\end{corollary}

\begin{remark} Corollary~\ref{invariant} is proved in Theorem 2.3 in \cite{ABCD} and the case for $n=2$ is proved in Theorem 3.3 in \cite{FMMV}.
\end{remark}

\begin{theorem}\label{inclusiontheorem}
	For a vertex $v_0$ in $K$, define $i_j:K^{n-1}\rightarrow K^n$ by $i_j(\sigma_1,...,\sigma_{n-1})=(\sigma_1,...,\sigma_{j-1},v_0,\sigma_j,...,\sigma_{n-1})$. Then we have $\scat(K^{n-1})=\SD(i_1,...,i_n).$
\end{theorem}

\begin{proof}
	Assume that $\scat(K^{n-1})=k.$ Then there exist categorical subcomplexes $L_0,...,L_k$ of $K^{n-1}$ covering $K^{n-1}$.
	More precisely, the inclusion map $\overline{\iota_j}:L_j\rightarrow K^{n-1}$ and the constant map $\overline{c_{v_0}}:L_j\rightarrow K^{n-1}$ are in the same contiguity class for all $j=0,...,k$, that is, $\overline{\iota_j}\sim\overline{c_{v_0}}$ on each $L_j$. Here, $\overline{\iota_j}=(pr_1,...pr_{n-1})$ where $pr_i:K^{n-1}\rightarrow K$ is the projection to the i-th factor and $\overline{c_{v_0}}=(c_{v_0},c_{v_0},...,c_{v_0})$ is the constant map. Since $K$ is edge-path connected and by Lemma~\ref{soru}, one can choose all the constant maps as the constant map $c_{v_0}:K^{n-1}\rightarrow K$. So, $pr_1\sim c_{v_0},$ $pr_2 \sim c_{v_0},$ $\ldots,$ $pr_{n-1}\sim c_{v_0}$. \\
	
	On the other hand, 
	
	$$i_1:=(c_{v_0},pr_1,...,pr_{n-1}), $$
	$$i_2:=(pr_1,c_{v_0},...,pr_{n-1}),$$ 
	$$...$$ 
	$$i_n:=(pr_1,pr_2,...,pr_{n-1},c_{v_0}). $$
	
	
	Since from the above paragraph
	
    $$
	pr_1|_{L_j}\sim pr_2|_{L_j}\sim ... \sim pr_{n-1}|_{L_j}\sim c_{v_0}|_{L_j},$$
	
we have  
$$
	i_1|_{L_j}\sim i_2|_{L_j}\sim ... \sim i_n|_{L_j}.
	$$ 
	   
	Let us prove the other way around. Suppose that $\SD(i_1,...,i_n)=k$. Then there exist subcomplexes $L_0,...,L_k$ of $K$ covering $K$ such that $i_1|_{L_j}\sim...\sim i_n|_{L_j}$ for all $j$. Our aim is to show that the inclusion map  $\overline{\iota_j}:L_j\rightarrow K^{n-1}$ and the constant map $\overline{c_{v_0}}:L_j\rightarrow K^{n-1}$ are in the same contiguity class for all $j=0,...,k$. Let $pr_i:K^{n-1}\rightarrow K$ be the projection map to the $i$-th factor for $i=1,...,n-1$. Then $$
	pr_1\circ i_1|_{L_j}\sim pr_1\circ i_2|_{L_j},$$ $$ pr_2\circ i_2|_{L_j}\sim pr_2\circ i_3|_{L_j},$$ $$...$$ $$ pr_{n-1}\circ i_{n-1}|_{L_j}\sim pr_{n-1}\circ i_n|_{L_j} $$
	
	so that $(pr_1,...,pr_{n-1})\sim(c_{v_0},c_{v_0},...,c_{v_0}).$ Therefore $\overline{\iota_j} \sim  \overline{c_{v_0}}$.
	
\end{proof}

In the previous theorem, we studied the relation between the simplical LS category and the higher contiguity distance of the inclusion maps from $K^{n-1}$ to $K^n$, and as a result, we found that $\scat(K^{n-1})=\SD(i_1,...,i_n)$. On the other hand, one can find another equality saying that $\scat(K)=\SD(\tilde{i}_1,...,\tilde{i}_n)$ where the inclusion maps $\tilde{i_j}$ are from $K$ to $K^n$ as introduced in the following theorem.

\begin{theorem}\label{thm2.2}
	Let $v_0$ be a vertex of the simplicial complex K. For the simplicial maps $\tilde{i}_j:K\rightarrow K^n$, which take $\sigma$ to the $n$-tuple whose $j$-th factor is $\sigma$ whereas the other factors are $v_0$, we have $\scat(K)=\SD(\tilde{i}_1,...,\tilde{i}_n).$
\end{theorem}

\begin{proof}
	Suppose that $\scat(K)=k.$ Then there exist categorical subcomplexes $L_0,...,L_k$ of $K$ covering $K$, that is, the inclusion map $\iota:L_i\hookrightarrow K$ and the constant map $c_{v_0}:L_i\rightarrow K$ are in the same contiguity class for all $i=0,1,\ldots,k$.
	
   We can write $\tilde{i}_j|_{L_i}$ as a composition as follows, for each $i$ and $j$:
   
	$$ \tilde{i}_1|_{L_i}=(\iota,c_{v_0},\ldots, c_{v_0})\circ \Delta_{L_i} $$
	$$ \tilde{i}_2|_{L_i}=(c_{v_0},\iota, c_{v_0},\ldots, c_{v_0})\circ \Delta_{L_i} $$
	 $$\ldots$$ 
	 $$ \tilde{i}_n|_{L_i}=(c_{v_0},\ldots, c_{v_0},\iota)\circ \Delta_{L_i} $$
	 
	 where $\Delta_{L_i}$ is the diagonal map for $L_i$.
	 
	 Since $L_i$ is categorical, $\iota\sim c_{v_0}$ on each $L_i$. Hence 
	 
	 $$ (\iota, c_{v_0},\ldots, c_{v_0})\sim (c_{v_0}, c_{v_0},\ldots, c_{v_0}) $$
	 $$ (c_{v_0},\iota, c_{v_0},\ldots, c_{v_0}) \sim (c_{v_0}, c_{v_0},\ldots, c_{v_0}) $$
	 $$\ldots$$ 
	 $$ (c_{v_0},\ldots, c_{v_0},\iota) \sim (c_{v_0}, c_{v_0},\ldots, c_{v_0}). $$
	 
	 Therefore $\tilde{i}_1|_{L_i}\sim \tilde{i}_2|_{L_i}\sim ... \sim \tilde{i}_n|_{L_i}$ which proves our claim. 
	 
	 On the other hand, suppose that $\SD(\tilde{i}_1,\ldots,\tilde{i}_n)=k.$ Then there exist subcomplexes $L_0,...,L_k$ of $K$ covering $K$ such that $\tilde{i}_1|_{L_i}\sim\ldots\sim \tilde{i}_n|_{L_i}$ for all $i$. Let us consider the first projection $p_1:K^n\rightarrow K$, then $$p_1\circ \tilde{i}_1|_{L_i}\sim p_1\circ \tilde{i}_2|_{L_i}\sim...\sim p_1\circ \tilde{i}_n|_{L_i}. $$
	 
	Therefore, we obtain $\iota\sim c_{v_0}$ which completes the proof.
	 
\end{proof}

\begin{corollary} The simplicial Lusternik-Schnirelmann category is an invariant of strong homotopy type.
\end{corollary}

\begin{proof} This follows from Theorem~\ref{maininvariant} and Theorem~\ref{thm2.2}. 

\end{proof}
The continuous version of Theorem~\ref{inclusiontheorem} can be realised as a discretisation of Theorem 3.4 in \cite{BV} whereas Theorem~\ref{thm2.2} is not yet given in the literature. So we introduce its continuous correspondence as Theorem~\ref{continuous}.

\begin{theorem}\label{continuous} Fix a point $x_0$ in $X$ and consider the inclusions $j_i:X\rightarrow X^n$ given by $j_i(x)=(x_0^1,...,x_0^{i-1},x,x_0^{i+1},...,x_0^{n})$ for $i=1,...,n$, here the superscripts of $x_0$'s stand for to emphasize in which factor $x_0$'s belong to. Then $D(j_1,...,j_n)=\cat(X).$
\end{theorem}

\begin{proof}
	Let $D(j_1,...,j_n)=k.$ Then there exist open subsets $U_0,...,U_k$ of X covering X such that $j_1|_{U_s}\simeq...\simeq j_n|_{U_s}$ for all $s=0,...,k.$ Assume that for each $s=0,1,\ldots,k$, there exists a homotopy $H_{s}^{i}:U_s\times I\rightarrow X^n$ such that $H_{s}^{i}(x,0)=j_i(x)$ and $H_{s}^{i}(x,1)=j_{i+1}(x)$ for each $i=1,...,n-1.$ We want to show that $id|_{U_s}\simeq c|_{U_s}$ for $s=0,...,k$ where $id:X\rightarrow X $ is the identity map and $c:X\rightarrow X$ is a constant map.
	For each $s=0,\ldots,k$, define $F_{s}^{i}:U_s\times I\rightarrow X$ by $F_{s}^{i}(x,t)=pr_i\circ H_{s}^{i}(x,t)$ where $pr_i:X^n\rightarrow  X$ is the projection map into the $i$-th factor and which satisfies 
\[
	F_{s}^{i}(x,0)=pr_i\circ H_{s}^{i}(x,0)=pr_i\circ j_i(x)=x, 
\]	
\[ 
	F_{s}^{i}(x,1)=pr_i\circ H_{s}^{i}(x,1)= pr_i\circ j_{i+1}(x)=x_0^{i}. 
\]

	Thus, $\cat(X)\leq D(j_1,...,j_n)$. \\

	On the other hand, assume that $\cat(X)=k$. There exists a categorical cover $U_0,...,U_k$ of X such that the inclusion map $\iota:X\rightarrow X$ and the constant map $c:X\rightarrow X $ are homotopic. For simplicity, let us write $U$'s without indices. Consider the homotopy $F:U\times I \rightarrow X,$ satisfying $F(x,0)=x$ and $F(x,1)=x_0$. Define the map $H^i:U\times I\rightarrow X$ by 
	$$ H^i(x,t)=
	\begin{cases}
		j_i(F(x,2t)), &\textrm{if  } 0\leq t\leq\frac{1}{2} \\
	j_{i+1}(F(x,2-2t)), &\textrm{if  } \frac{1}{2}\leq t\leq 1,
	\end{cases}
	$$ so that $H^i(x,0)=j_i(x)$ and $H^i(x,1)=j_{i+1}(x)$ are satisfied. Since $$
	H^i(x,1/2)=j_i(F(x,1))=j_i(x_0)=(x_0,x_0,...,x_0)=j_{i+1}(F(x,1)), $$ $H^i$ is continuous for each $i=1,\ldots,n-1$. Hence $D(j_1,...,j_n)\leq \cat(X)$.

\end{proof}

\begin{proposition}\label{prop38}
	Let $K,K'$ and $K''$ be simplicial complexes, $\eta,\eta':K''\rightarrow K$ and $\varphi_1,\dots,\varphi_n:K\rightarrow K'$ be simplicial maps. If $\varphi_1\circ\eta'\sim\dots\sim\varphi_n\circ\eta'$, then $\SD(\varphi_1\circ\eta,\dots,\varphi_n\circ\eta)\leq \SD(\eta,\eta')$.
\end{proposition}

\begin{proof}
	Assume that $\SD(\eta,\eta')=k$. There exist subcomlexes $K_0,\dots,K_k$ of $K$ such that $\eta|_{K_i}\sim\eta'|_{K_i}$ for $i=1,\dots,k$. We have 
	$$ \varphi_1\circ\eta|_{K_i}\sim \varphi_1\circ\eta'|_{K_i}$$
	$$ \varphi_2\circ\eta|_{K_i}\sim \varphi_2\circ\eta'|_{K_i}$$
	$$\dots$$
	$$ \varphi_n\circ\eta|_{K_i}\sim \varphi_n\circ\eta'|_{K_i}$$
	for all $i=0,\dots,k.$ Since $\varphi_1\circ\eta'\sim\dots\sim\varphi_n\circ\eta'$, we conlude that $\varphi_1\circ\eta|_{K_i}\sim \varphi_2\circ\eta|_{K_i} \sim \dots \sim \varphi_n\circ\eta|_{K_i}$. Hence $\SD(\varphi_1\circ\eta,\dots,\varphi_n\circ\eta)\leq k$.
  
\end{proof}

\begin{corollary}\label{yeni2} If $\varphi_1, \ldots, \varphi_n: K\rightarrow K'$ are arbitrary simplicial maps, then $\SD(\varphi_1, \ldots, \varphi_n)\leq \scat(K)$.
\end{corollary}

\begin{proof}
	If we take $K''=K, \eta=id_K$ and $\eta'=c_{v_0}$ in Proposition \ref{yeni2} where $id_K$ is the identity map and $c_{v_0}$ is the constant map. By Proposition \ref{yeni2}, we have
	$$ \SD(\varphi_1,\dots,\varphi_n)=\SD(\varphi_1\circ id_K,\dots,\varphi_n\circ id_K)\leq \SD(id_K,c_{v_0})=scat(K) $$
\end{proof}

The importance of Corollary~\ref{yeni2} comes from the following fact. One can say that the contiguity distance of arbitrary two maps can be used to find a lower bound for simplicial LS-category, that is, for two simplicial maps $\varphi$ and $\psi$ from $K$, one has $\SD(\varphi,\psi)\leq \scat(K)$. While this proposition (as introduced in \cite{BPV}) is taken account, Corollary~\ref{yeni2} gives a better lower bound for $\scat(K)$ when one takes more arbitrary simplicial maps from $K$ (due to Proposition~\ref{prop3.4}). 

\begin{theorem}\label{strongly}
	If $K$ is strongly collapsible if and only if $\SD(\varphi_1,\dots,\varphi_n)=0$ for any simplicial maps $\varphi_1,\dots,\varphi_n:K\rightarrow K'.$
\end{theorem}

\begin{proof}
	Assume that $K$ is strongly collapsible, then we have $\scat(K)=0$, more precisely, $id_K\sim c_{v_0}$. By Theorem \ref{thm2.2}, $\SD(\tilde{i}_1,\dots,\tilde{i}_n)=0$ where for each $j$, $\tilde{i}_j:K\rightarrow K^n$ be the simplicial map as defined in Theorem \ref{thm2.2}. Since $id_k\sim c_{v_0}$, we obtain $\tilde{i}_j\sim \tilde{c}_{v_0}$ for all $j=1,\dots,n$ where $\tilde{c}_{v_0}:=\Delta_K\circ c_{v_0}$ and $\Delta_K$ is the diagonal map of K. Consider the following composition of maps
	
	\begin{displaymath}
		\xymatrix{
			K \ar@/^/[r]^{\tilde{c}_{v_0}}
			\ar@/_/[r]_{\tilde{i}_k} &
			K^n \ar[r]^{\tilde{\varphi_i}} & (K')^n
		}
	\end{displaymath}

\begin{displaymath}
	\xymatrix{
		K \ar@/^/[r]^{\tilde{c}_{v_0}}
		\ar@/_/[r]_{\tilde{i}_k} &
		K^n  \ar[r]^{\tilde{\varphi_j}} & (K')^n
	}
\end{displaymath}

where $\tilde{\varphi_i}:K^n\rightarrow (K')^n$ defined by $(\sigma_1,\dots,\sigma_n)\mapsto (\varphi_i(\sigma_1),\dots,\varphi_i(\sigma_n))$ for all $i=1,\dots,n$. Then for all $i,k$, we have $\tilde{\varphi}_i\circ \tilde{c}_{v_0} \sim \tilde{\varphi}_i \circ \tilde{i}_k$ and for all $j,k$, we have $\tilde{\varphi}_j\circ \tilde{c}_{v_0} \sim \tilde{\varphi}_j \circ \tilde{i}_k$. Since $K'$ is edge-path connected, all constant maps are in the same contiguity class. Therefore $\tilde{\varphi}_i \sim \tilde{\varphi}_{i+1}$ for all $i=1,\dots,n-1$. By the definition of $\tilde{\varphi_i}$, we conclude that $\varphi_i\sim \varphi_{i+1}$ for all $i=1,\dots,n-1$. Hence $\SD(\varphi_1,\dots,\varphi_n)=0$.

On the other hand, suppose that $\SD(\varphi_1,\dots,\varphi_n)=0$. If we choose $\varphi_j=\tilde{i}_j$ for each $j=1,\dots,n$ we conclude that $\scat(K)=0$. This is the same as $K$ to be strongly collapsible.
\end{proof}

\begin{example} Consider the abstract simplicial complex $K$ as given in Figure~\ref{f1}. Consider any arbitrary simplicial maps 

\[
\varphi_1, \ldots, \varphi_n:K\rightarrow K.
\]

If one cover $K$ by the subcomplexes which do not contain any loops, then it is needed at least 3 subcomplexes. This is due to the fact that in order not to draw a loop, the number of edges must be 1 less then the number of vertices. 

Also, one can cover $K$ by two subcomplexes which has two loops in total. More precisely, if one can cover $K$ by the subcomplexes $U_0$ and $U_1$, then in the best case, either $U_0$ and $U_1$ both have one loop or $U_0$ has no loop while $U_1$ has two loops. Explicity writing, if $U_0$ and $U_1$ both have one loop, then they both contains 5 vertices and 5 edges. If $U_0$ has no loop while $U_1$ has 2 loops, then $U_0$ has 5 vertices and 4 edges while $U_1$ has 5 vertices and 6 edges. Here, $U_1$ has indeed 2 loops due to the fact that in order to draw the last edge, we need to draw it between non-adjacent vertices.  

If we consider the subcomplexes $U_0$ and $U_1$ as described in the preceding paragraph, then the least number of strongly collapsible subcomplexes of both $U_0$ and $U_1$ is 3. Then from the strongly collapsibility and by Theorem~\ref{strongly}, we get $\SD(\varphi_1, \ldots, \varphi_n)\leq 2$. 

On the other hand, if one has no loops in the subcomplexes as described in the first paragraph, then from the strongly collapsibility and by Theorem~\ref{strongly}, we get $\SD(\varphi_1, \ldots, \varphi_n)\leq 2$. 


\end{example}

\begin{proposition}
	If $K'$ is strongly collapsible then $\SD(\varphi_1,\dots,\varphi_n)=0$ for any simplicial maps $\varphi_1,\dots,\varphi_n:K\rightarrow K'.$
\end{proposition}

\begin{proof}
	Assume that $K'$ is strongly collapsible, then we have $id_{K'}\sim c_{v_0}$ where $id_{K'}:K'\rightarrow K'$ is the identity map and $c_{v_0}:K'\rightarrow K'$ is the constant map with $v_0$ being a vertex in $K'$. Consider the following compositions
		\begin{displaymath}
		\xymatrix{
		K \ar[r]^{\varphi_i} & K' \ar@/^/[r]^{id_{K'}} \ar@/_/[r]_{c_{v_0}} & K'
		}
	\end{displaymath}
	
	\begin{displaymath}
		\xymatrix{
			K \ar[r]^{\varphi_j} & K' \ar@/^/[r]^{id_{K'}} \ar@/_/[r]_{c_{v_0}} & K'
		}
	\end{displaymath}
	where $i,j=1,\dots,n.$ Then for all $i,j$, we have $id_{K'}\circ\varphi_i\sim c_{v_0}\circ\varphi_i$ and  $id_{K'}\circ\varphi_j\sim c_{v_0}\circ\varphi_j$. Since all constant maps are in the same contiguity class, we come to conclusion that $\varphi_i\sim\varphi_j$ for all $i,j=1,\dots,n.$ Therefore $\SD(\varphi_1,\dots,\varphi_n)=0$.
	
\end{proof}

\begin{example} 	Consider the abstract simplicial complex $K=\{\{0\},\{1\},\{2\},\{0,1\},\{0,2\},\{1,2\}\}$ as given in Figure~\ref{f2} which is not strongly collapsible. 

\begin{center}
\begin{tikzpicture}[scale=0.5]\label{f2}
    \coordinate[label=below:$0$] (0) at (0,0);
    \coordinate[label=above:$1$] (1) at (2,3.4);
    \coordinate[label=below:$2$] (2) at (4,0);
    
    \draw (0) -- (1);
    \draw (0) -- (2);
    \draw (2) -- (1);
    
    \foreach \i in {0,1,2} {
        \filldraw[black] (\i) circle (3pt);
    }
\end{tikzpicture}
\end{center}
\[
\textit{Figure 3.2}
\]

Consider the identity map $\id:K\rightarrow K$, the constant map $c_0$ on $\{0\}$ and the simplicial map 
\[
\varphi: K\rightarrow K \hspace{0.05in} \text{by }\hspace{0.05in} \varphi(\{0\})=\{1\}, \hspace{0.03in}  \varphi(\{1\})=\{2\}, \hspace{0.03in} \varphi(\{2\})=\{0\}.
\] 

Let us take $U_0=\{\{0\},\{1\},\{2\},\{0,1\},\{0,2\}\}$ and $U_1=\{\{1\},\{2\},\{1,2\}\}$. On $U_0$, $c_0\sim\varphi$ since that $c_0 \sim_c c_1 \sim_c \varphi$ where $c_1:K\rightarrow K$ is the constant map on $\{1\}$; and on $U_1$, $c_0\sim\varphi$ since that they are also contiguous. Moreover, since that $U_0$ and $U_1$ are strongly collapsible complexes, as mentioned in \cite{BPV}, $\id|_{U_0}\sim {c_0}|_{U_0}$ and $\id|_{U_1}\sim {c_0}|_{U_1}$.

On the other hand, $K$ is not strongly collapsible, so we conclude that $\SD(\id,c_0,\varphi)=1$.

\end{example}

Next, we will focus on the relation between the higher contiguity distance of simplicial maps and the higher homotopic distance of their geometric realisations. 

\begin{lemma}\cite{S} \label{spanierlem} If contiguous simplicial maps agree on a subcomplex $L$, then they determine continuous maps which are homotopic relative to $|L|$.

\end{lemma}

\begin{proposition}\label{abcde}
	For simplicial maps $\varphi_1,\varphi_2,...,\varphi_n:K\rightarrow L$, we have 
\[
\D(|\varphi_1|,...,|\varphi_n|)\leq \SD(\varphi_1,...,\varphi_n).
\]
\end{proposition}

\begin{proof}
	Let $ \SD(\varphi_1,...,\varphi_n)=k$. Then there are subcomplexes $K_0,K_1,...,K_k$ covering $K$ such that $\varphi_i|_{K_j}\sim \varphi_{i-1}|_{K_j}$ for all $i=2,...,n$ and $j=0,...,k$. By Lemma \ref{spanierlem}, we have $|\varphi_i|\big| _{|K_j|}\simeq |\varphi_{i-1}|\big| _{|K_j|} $  for all $i=2,...,n$ and $j=0,...,k$.
	Moreover the union of the closed subsets $|K_0|,|K_1|,...,|K_n|$ cover $|K|$.
\end{proof}

The following is an example for the strict case of the inequality of Proposition~\ref{abcde}.

\begin{example}
	Consider the simplicial complex $K$ given in Figure~\ref{pic}. Let $\tilde{i}_j:K\rightarrow K^n$ be the simplicial maps as defined in Theorem \ref{thm2.2}. By Example 3.3 in \cite{FMV} $\scat(K)=1$ and by Theorem \ref{thm2.2} we have $\SD(\tilde{i}_1,...,\tilde{i}_n)=1$. Since $K$ is contractible, $0=\cat(|K|)=\D(|\tilde{i}_1|,...,|\tilde{i}_n|)$. Hence we conclude that $\D(|\tilde{i}_1|,...,|\tilde{i}_n|) < \SD(\tilde{i}_1,...,\tilde{i}_n)$.

\[
\begin{tikzpicture}[scale=0.5] \label{pic}
 
\draw [fill=gray, gray] (0,0) -- (0,0) -- (4,6.8) -- (8,0);
\draw (0,0) -- (3,2.8);
\draw (0,0) -- (4,1);
\draw (4,6.8) -- (3,2.8);
\draw (4,6.8) -- (5,2.8);
\draw (8,0) -- (4,1);
\draw (8,0) -- (5,2.8);
\draw (3,2.8) -- (4,1);
\draw (3,2.8) -- (5,2.8);
\draw (5,2.8) -- (4,1);
\draw (0,0) -- (4,6.8);
\draw (0,0) -- (8,0);
\draw (8,0) -- (4,6.8);
\end{tikzpicture}
\]
\[
\textit{Figure 3.3}
\]

\end{example}

\section{Barycentric Subdivision}

For a given abstract simplicial complex $K$,  recall that the (first) barycentric subdivision of $K$, denoted by $\sd{K}$, is defined as follows.

The vertex set of $\sd{K}$ is the set of simplices of $K$ and the set $\{\sigma_1,\ldots,\sigma_q\}$ determines a simplex in $\sd{K}$ if $\sigma_1,\ldots,\sigma_q$ are simplices of $K$ satisfying $\sigma_1 \subset \ldots \subset \sigma_q$. 

\begin{definition}\cite{FMV} For a simplicial map $\varphi:K\rightarrow K'$, the induced map on barycentric subdivisions $\sd\varphi:\sd(K)\rightarrow \sd(K')$ is defined as 
\[
(\sd\varphi)(\{\sigma_1,\ldots,\sigma_q\})=\{\varphi(\sigma_1),\ldots,\varphi(\sigma_q)\}.
\]
\end{definition}

Notice that $\sd\varphi$ is a simplicial map.

\begin{remark}\cite{FMV} $\sd(\varphi\circ \psi)=\sd\varphi\circ \sd\psi$ and $\sd (\id)=\id$.
\end{remark}

\begin{proposition} \cite{FMV} \label{prop}
	For simplicial maps $\varphi, \psi: K\rightarrow L$, if $\varphi$ and $\psi$ are in the same contiguity class then $\sd(\varphi)$ and $\sd(\psi)$ are in the same contiguity class.
\end{proposition}

\begin{theorem}\label{sdt}
	For simplicial maps $\varphi_1,\dots,\varphi_n:K\rightarrow K'$, $\SD(\sd(\varphi_1),\dots,\sd(\varphi_n))\leq \SD(\varphi_1,\dots,\varphi_n)$.
\end{theorem}

\begin{proof}
	Assume that $\SD(\varphi_1,\dots,\varphi_n)=k$. There exist subcomplexes $K_0,\dots,K_k$ covering $K$ such that $\varphi_1|_{K_j}\sim \varphi_2|_{K_j} \sim \dots \sim \varphi_n|_{K_j}$ for all $j=0,\dots,k$.
	Let consider the cover $\{\sd(K_0),\sd(K_1),\dots,\sd(K_k)\}$ of $\sd(K)$. By Proposition~\ref{prop}, if $\varphi_1|_{K_j}\sim \varphi_2|_{K_j} \sim \dots \sim \varphi_n|_{K_j}$, then $\sd(\varphi_1|_{K_j})\sim \sd(\varphi_2|_{K_j}) \sim \dots \sim \sd(\varphi_n|_{K_j}).$
	
	If we show that $\sd(\varphi_i|_{K_j})=\sd(\varphi_i|_{\sd(K_j)})$ for all $i=1,\dots,n$ and $j=0,\dots,k$, then we are done. For $\{\sigma_1,\dots,\sigma_q\} \in \sd(K_j)$, we have
	\begin{eqnarray*}
	 \sd(\varphi_i|_{K_j})(\{\sigma_1,\dots,\sigma_q\})&=&\{\varphi_i|_{K_j}(\sigma_1),\dots,\varphi_i|_{K_j}(\sigma_q)\} \\
	&=&\{\varphi_i(\sigma_1),\dots,\varphi_i(\sigma_q)\} \\
	&=&\{\varphi_i|_{\sd(K_j)}(\sigma_1),\dots,\varphi_i|_{\sd(K_j)}(\sigma_q)\} \\
	&=&\sd(\varphi_i|_{\sd(K_j)}(\{\sigma_1,\dots,\sigma_q\})).
\end{eqnarray*}

where the first and the last equalities follow from the definition of the induced map on the barycentric subdivisions whereas the third equality comes from the fact that $\sigma_s$ is a simplex in $\sd(K_j)$ for $s=1,\ldots,q$. 

Consequently, since the following holds $$\sd(\varphi_1|_{K_j}) \sim \dots \sim \sd(\varphi_n|_{K_j}),$$ we obtain $$\sd(\varphi_1|_{\sd(K_j)})\sim \dots \sim \sd(\varphi_n|_{\sd(K_j)}).$$ for all $j=0,\dots,k.$

\end{proof} 

The following is an example for the strict case of Theorem~\ref{sdt}.

\begin{example} Consider the abstract simplicial complex $K$ as given in Figure~\ref{f1}. In Example 3.6 in \cite{FMMV2}, it is showed that $\scat(K)=2$ while $\scat(\sd{K})=1$. For a fixed vertex $v_0$ in $K$, consider the simplicial inclusion maps $i_1,i_2,i_3:K\rightarrow K^3$ given by $i_1(v)=(v,v_0,v_0)$, $i_2(v)=(v_0,v,v_0)$ and $i_1(v)=(v_0,v_0,v)$. Hence, by Theorem~\ref{thm2.2}, we have 
\[
1=\SD(\sd{i_1},\sd{i_2},\sd{i_3})< \SD(i_1,i_2,i_3)=2.
\]
\end{example}

\section{Acknowledgement}
The authors would like to acknowledge that this paper is submitted in partial fulfilment of the requirements for PhD degree at Bursa Technical University.\\

\end{document}